\documentclass[11pt,a4paper,reqno]{amsart}

\usepackage[T1]{fontenc}
\usepackage[utf8]{inputenc}
\usepackage{amsmath,amssymb,amsfonts,amsthm,mathtools}
\usepackage{enumitem}
\usepackage[margin=1.05in]{geometry}
\usepackage[colorlinks=true,linkcolor=blue,citecolor=blue,urlcolor=blue]{hyperref}
\usepackage{microtype}

\numberwithin{equation}{section}
\newtheorem{theorem}{Theorem}[section]
\newtheorem{lemma}[theorem]{Lemma}
\newtheorem{proposition}[theorem]{Proposition}
\newtheorem{corollary}[theorem]{Corollary}

\theoremstyle{remark}
\newtheorem{remark}[theorem]{Remark}

\newcommand{\dd}{\,\mathrm d}
\newcommand{\1}{\mathbf 1}
\newcommand{\E}{\mathbb E}
\newcommand{\Prob}{\mathbb P}
\newcommand{\cH}{\mathcal H}
\newcommand{\cE}{\mathcal E}
\newcommand{\cC}{\mathcal C}

\renewcommand{\Cap}{\operatorname{Cap}}
\newcommand{\Ceff}{\operatorname{C}_{\mathrm{eff}}}

\allowdisplaybreaks
\title[Sharp integral Fujita criteria on graphs]{Flow Decomposition and Sharp Integral Fujita Criteria on Weighted Graphs}
\author{Qingsong Gu}
\address{School of Mathematics, Nanjing University, Nanjing 210093, P. R. China} \email{qingsonggu@nju.edu.cn}

\author{Lu Hao}
\address{Universit\"{a}t Bielefeld, Fakult\"{a}t f\"{u}r Mathematik, Postfach 100131, D-33501, Bielefeld, Germany}
\email{lhao@math.uni-bielefeld.de}

\author{Xueping Huang}
\address{School of Mathematics and Statistics, Nanjing University of Information Science and Technology,
	Nanjing 210044, P. R. China}
\email{hxp@nuist.edu.cn}
	
\author{Yuhua Sun}
\address{School of Mathematical Sciences and LPMC, Nankai University, 300071
		Tianjin, P. R. China}
\email{sunyuhua@nankai.edu.cn}

\thanks{\noindent
 Q. Gu was supported by the National Natural Science Foundation of China (Grant Nos. 12101303 and 12171354).
 L. Hao was funded by the Deutsche Forschungsgemeinschaft (DFG, German Research Foundation), Project-ID 317210226, SFB 1283.
 X. Huang was supported by
	the National Natural Science Foundation of China (Grant No. 11601238).
 Y. Sun was funded by the National Natural Science Foundation of
	China (Grant No. 12371206) and the Fundamental Research Funds for the Central Universities, No. 050-63263078.}


\begin{document}

\begin{abstract}
We study the Fujita phenomenon for semilinear heat inequalities generated by
variable-speed Laplacians on infinite weighted graphs.  Assuming that the
graph carries a proper adapted path metric, we establish an integral
volume-growth criterion forcing every nonnegative global classical
supersolution on the open cylinder $(0,\infty)\times V$ to vanish, without
assuming an initial value or trace.  We prove that a nontrivial supersolution
of this kind exists if and only if the equation has a positive global Cauchy
solution for some nonzero point-source datum.  A complementary heat-kernel
construction gives global Cauchy solutions for all sufficiently small
point-source data when the same volume integral converges and a matching
anchored heat-kernel upper bound is available.  This proves sharpness on
integer lattices and on a family of logarithmically perturbed weighted
half-lines; in the latter examples, even the exponent of an iterated logarithm
can determine the existence--nonexistence alternative.  A finer nonexistence
criterion couples intrinsic volume growth with the capacity of intrinsic
annuli.  Its proof combines parabolic testing, a Laplace--resolvent reduction,
and a pathwise decomposition of resolvent currents.  The nonexistence results
require no volume-doubling property, Poincar\'e inequality, heat-kernel bound,
or stochastic completeness.
\end{abstract}

\maketitle
\tableofcontents

\section{Introduction and main results}

The Fujita problem asks for which exponents a nontrivial nonnegative solution of a
semilinear heat equation can exist for all positive times.  Fujita considered
the Cauchy problem
\[
\begin{cases}
        \partial_tu-\Delta u=u^q,
        & (t,x)\in(0,\infty)\times\mathbb R^N,\\
        u(0,x)=u_0(x)\ge0,\quad u_0\not\equiv0,
        & x\in\mathbb R^N,
\end{cases}
\]
He proved that no nontrivial nonnegative initial datum yields a global
solution when $1<q<1+2/N$, whereas sufficiently small initial data do yield
global solutions when $q>1+2/N$, thereby identifying the critical exponent
\[
        q_F=1+\frac2N;
\]
see \cite{Fujita66}.  The borderline case $q=q_F$ was subsequently completed
in \cite{Hayakawa73,KST77}.  Thus, for a fixed dimension $N$, the critical
threshold is expressed in terms of the exponent $q$; dually, for a fixed
$q>1$, the corresponding critical dimension is
\[
        N_F=\frac2{q-1}.
\]

On a complete noncompact Riemannian manifold, a single dimension no longer
captures the relevant large-scale geometry.  A natural substitute is the
volume-growth function
\[
        V_M(r):=\operatorname{Vol}(B(o,r)).
\]
Fujita-type nonexistence results on manifolds have accordingly been formulated
in terms of polynomial and critical polynomial--logarithmic volume bounds;
see, for example, \cite{Zhang99,MMP17}.  Motivated by this geometric viewpoint,
the criterion established below has the same form, with $V_M$ replaced by the
intrinsic weighted-ball volume $M_o$ defined below:
\[
        \int_1^\infty\frac{r}{M_o(r)^{q-1}}\,\dd r=\infty.
\]
When $M_o(r)\asymp r^D$, this condition is equivalent to
$q\le1+2/D$ and therefore recovers the classical Euclidean nonexistence
range.

The purpose of this paper is to establish this nonexistence criterion for a
graph equipped with a positive vertex weight and to verify its sharpness on
natural classes of examples.  The nonexistence theorem is unconditional
beyond the proper adapted-metric assumption. The complementary existence statement is restricted to settings where the additional regularity holds, a condition that is furnished by an anchored heat-kernel upper bound.

A time-domain distinction should be made explicit.  Fujita's original problem
and the graph problem in \cite{LinWu17} are Cauchy problems: the solution is
defined up to $t=0$, an initial datum is prescribed there, and the differential
equation is imposed for $t>0$.  Although no particular datum is fixed in the
very weak graph formulation of \cite{MPS26}, the solution is again defined up
to $t=0$ and its initial value is retained as a boundary term in the weak
formulation.  By contrast, our nonexistence theorem is formulated on the open
cylinder $(0,\infty)\times V$ and assumes neither a value nor a trace at the
initial time.  As regards the initial-time requirement, it is therefore a
trace-free Liouville statement for an \emph{a priori} larger class.  This
distinction does not create a mismatch at the existential level relevant to
the Fujita threshold:
Proposition~\ref{prop:supersolution-solution-equivalence} shows that a
nontrivial global supersolution on the open cylinder exists if and only if
the equation has a
positive global Cauchy solution for some nonzero point-source datum.  Thus the
critical-threshold comparison is exact at this existential level, while the
nonexistence statement itself remains independent of initial data.  This
equivalence does not assert global solvability for every small datum; the
converse via a heat kernel condition below gives the precise small-data class used here.

Let $(V,E)$ be an infinite, connected, locally finite graph.  We allow at most
one edge between two distinct vertices and no loops.  Let
$\mu:V\times V\to[0,\infty)$ be a
symmetric edge weight satisfying
\[
        \mu_{xy}=\mu_{yx}>0\quad\text{if }x\sim y,
        \qquad
        \mu_{xy}=0\quad\text{otherwise}.
\]
For $x\in V$, put
\[
        \mu(x):=\sum_{y\sim x}\mu_{xy},
\]
and, for $A\subset V$, write
\[
        \mu(A):=\sum_{x\in A}\mu(x).
\]
Let
\[
        \Delta f(x)
        :=\frac1{\mu(x)}\sum_{y\sim x}\mu_{xy}\bigl(f(y)-f(x)\bigr)
\]
be the standard normalized graph Laplacian.

Fix a vertex weight $\sigma:V\to(0,\infty)$ and define the reference measure
\[
        \nu(x):=\sigma(x)\mu(x),
        \qquad
        \nu(A):=\sum_{x\in A}\nu(x).
\]
The diffusion operator studied here is
\begin{equation}\label{eq:laplacian}
        \Delta_\sigma f(x)
        :=\frac1{\nu(x)}\sum_{y\sim x}\mu_{xy}\bigl(f(y)-f(x)\bigr)
        =\frac1{\sigma(x)}\Delta f(x).
\end{equation}
Thus the jump rate at $x$ is $\sigma(x)^{-1}$, and $\Delta_\sigma$ is symmetric
with respect to $\nu$.

The choice of the reference measure $\nu=\sigma\mu$, or equivalently the
placement of $\sigma$ in \eqref{eq:laplacian}, determines the linear
evolution.  Indeed, the parabolic inequality considered below can be written
equivalently as
\[
        \sigma\,\partial_tu-\Delta u\ge\sigma u^q.
\]
Its stationary supersolution satisfies
\[
        -\Delta u\ge\sigma u^q,
\]
and is precisely the
weighted Lane--Emden inequality studied in \cite{GHHS26}. 

The geometry appropriate to \eqref{eq:laplacian} is an intrinsic path metric.
A positive symmetric edge-length function $\rho:E\to(0,\infty)$ is called
$\nu$-adapted if
\begin{equation}\label{eq:adapted-weight}
        \sum_{y\sim x}\mu_{xy}\rho(x,y)^2\le\nu(x),
        \qquad x\in V.
\end{equation}
Such a length function always exists; for example,
\[
        \rho(x,y)=\sqrt{\min\{\sigma(x),\sigma(y)\}}
\]
is $\nu$-adapted.  Let $d_\rho$ be the associated path metric,
\[
 d_\rho(x,y)
 :=\inf\left\{\sum_{i=1}^n\rho(x_{i-1},x_i):
 x=x_0\sim x_1\sim\cdots\sim x_n=y\right\}.
\]
Throughout the paper we assume that $d_\rho$ is proper, meaning that every ball
\[
        B_\rho(o,r):=\{x\in V:d_\rho(o,x)\le r\}
\]
is finite.  In the present locally finite intrinsic-path setting, properness is
equivalent to metric completeness.  Finiteness of balls is the exact compactness
input needed in the finite-network argument.  We write
\[
        M_o(r):=\nu\bigl(B_\rho(o,r)\bigr).
\]
Throughout the paper, $c_q>0$ denotes a constant depending only on $q$; its
value may change from line to line.  For nonnegative quantities $A$ and $B$,
we write $A\lesssim B$ if $A\le CB$ for some constant $C>0$, and
$A\asymp B$ if both $A\lesssim B$ and $B\lesssim A$.  In the examples,
the implicit constants may depend on the fixed model parameters.

Let $p_\sigma(t,x,y)$ denote the minimal heat kernel of $\Delta_\sigma$ with
respect to $\nu$; equivalently, it is the monotone limit of the Dirichlet heat
kernels on finite connected exhaustions.  It is symmetric and sub-Markovian:
\[
        p_\sigma(t,x,y)=p_\sigma(t,y,x),
        \qquad
        \sum_y p_\sigma(t,x,y)\nu(y)\le1.
\]
No conservativeness assumption is made.

We denote by $\ell_0(V)$ the space of finitely supported
real-valued functions on $V$. For $f\in \ell_0(V)$, define
\[
\mathcal E_{\mathrm D}(f)
:=
\frac12\sum_{x\in V}\sum_{y\sim x}
\mu_{xy}\bigl(f(x)-f(y)\bigr)^2.
\]
For $\lambda>0$ and a nonempty finite set $A\subset V$, define the
$\lambda$-capacity associated with the reference measure $\nu$ by
\begin{equation}\label{eq:lambda-capacity-definition}
\Cap_\lambda^\nu(A)
:=
\inf\left\{
\mathcal E_{\mathrm D}(f)+\lambda\sum_{x\in V}f(x)^2\nu(x):
 f\in\ell_0(V),\quad f\ge1\text{ on }A
\right\}.
\end{equation}
By truncation, one may equivalently require $f=1$ on $A$ and $0\le f\le1$.
The mass term is $\lambda\nu$, as dictated by the resolvent of
$\Delta_\sigma$.

For $0<r<R$, define the intrinsic annular capacity
\begin{equation}\label{eq:annular-capacity-definition}
\begin{aligned}
\cC_{o,\rho}(r,R)
&:=
\Cap\bigl(B_\rho(o,r),V\setminus B_\rho(o,R)\bigr)\\
&:=
\inf\left\{
\mathcal E_{\mathrm D}(f):
\begin{array}{l}
 f\in\ell_0(V),\\[-2pt]
 f=1\text{ on }B_\rho(o,r),\\[-2pt]
 f=0\text{ on }V\setminus B_\rho(o,R)
\end{array}
\right\}.
\end{aligned}
\end{equation}
We abbreviate
\[
        \cC_{o,\rho}(r):=\cC_{o,\rho}(r,2r).
\]

We use a trace-free notion on the open cylinder in the nonexistence part.
A nonnegative global classical supersolution of
\begin{equation}\label{eq:heat-ineq}
        \partial_tu-\Delta_\sigma u\ge u^q
        \qquad\text{on }(0,\infty)\times V
\end{equation}
is a finite nonnegative function such that $u(\cdot,x)\in C^1((0,\infty))$
for every $x\in V$ and \eqref{eq:heat-ineq} holds pointwise.  A positive global
classical solution of
\begin{equation}\label{eq:heat-equation}
        \partial_tu-\Delta_\sigma u=u^q
        \qquad\text{on }(0,\infty)\times V
\end{equation}
is a finite function $u:(0,\infty)\times V\to(0,\infty)$ such that
$u(\cdot,x)\in C^1((0,\infty))$ for every $x\in V$ and
\eqref{eq:heat-equation} holds pointwise.  Neither notion includes a value or
a trace at $t=0$.

For comparison with the standard Fujita formulation, let
$u_0:V\to[0,\infty)$ be finite-valued.  By a nonnegative global classical
solution of the Cauchy problem
\begin{equation}\label{eq:cauchy-problem}
\begin{cases}
        \partial_tu-\Delta_\sigma u=u^q,
        &t>0,\ x\in V,\\
        u(0,x)=u_0(x),
        &x\in V,
\end{cases}
\end{equation}
we mean a function $u:[0,\infty)\times V\to[0,\infty)$ such that
$u(\cdot,x)\in C([0,\infty))\cap C^1((0,\infty))$ for every $x\in V$ and the
stated equation holds pointwise for $t>0$.  Such a Cauchy solution is called
positive if $u(t,x)>0$ for all $t>0$ and $x\in V$; positivity at $t=0$ is not
required.  Thus a Cauchy solution is defined on the left-closed interval
$[0,\infty)$, but the differential equation is imposed only for $t>0$; the
endpoint $t=0$ carries the initial condition.

Every nonnegative global Cauchy solution restricts to a global classical
supersolution on $(0,\infty)$.  Consequently, every nonexistence theorem below
automatically rules out global Cauchy solutions for all nontrivial
nonnegative initial data.  Proposition~\ref{prop:supersolution-solution-equivalence}
provides the converse implication at the level of existence.

The main parabolic resolvent--capacity estimate is the following.

\begin{theorem}\label{thm:heat-capacity-lower}
Let $q>1$.  There exist a constant $c_q>0$, depending only on $q$, and an
absolute constant $\eta\in(0,1)$ such that, for every $o\in V$,
\begin{align}
\cH_{q,\sigma}(o)
&:=\int_0^\infty\sum_{x\in V}p_\sigma(t,o,x)^q\nu(x)\,\dd t \notag\\
&\ge
c_q\int_0^{\eta^2}\lambda^{q-2}
\int_0^{\eta\lambda^{-1/2}}
 r\,\Cap_\lambda^\nu\bigl(B_\rho(o,r)\bigr)^{1-q}\,\dd r\,\dd\lambda,
\label{eq:heat-resolvent-capacity-lower}\\
&\ge
c_q\int_1^\infty
\frac{r}{M_o(2r)^{q-1}}
\log\!\left(
1+\frac{M_o(2r)}{r^2\cC_{o,\rho}(r)}
\right)\,\dd r.
\label{eq:heat-capacity-volume-lower}
\end{align}
\end{theorem}

The first inequality is the parabolic capacity estimate.  Indeed, for $t>0$
and every nonempty finite set $A$,
\[
t\Cap_{1/t}^{\nu}(A)
=
\inf\left\{
\|f\|_{L^2(\nu)}^2+t\mathcal E_{\mathrm D}(f):
 f\in\ell_0(V),\ f\ge1\text{ on }A
\right\}.
\]
This is the optimal balance between $L^2(\nu)$ mass and Dirichlet energy at
time scale $t$, whereas $M_o(2r)+t\cC_{o,\rho}(r)$ is the explicit upper
bound obtained from an annular cutoff.  The absolute constant $\eta$ is a
harmless buffer ensuring the strict first-exit range
$r<\lambda^{-1/2}$.

The second estimate is obtained by retaining the full Laplace-parameter
integration and testing the $\lambda$-capacity with an intrinsic annular cutoff.
The proof in Section~4 yields the sharper intermediate estimate
\eqref{eq:heat-capacity-volume-phi-lower}.
The intrinsic radial cutoff satisfies
\begin{equation}\label{eq:intro-annular-capacity-cutoff}
        \cC_{o,\rho}(r)\le r^{-2}M_o(2r).
\end{equation}

\begin{theorem}\label{thm:capacity-volume-criterion}
Let $q>1$.  If, for some $o\in V$,
\begin{equation}\label{eq:capacity-volume-condition}
\int_1^\infty
\frac{r}{M_o(2r)^{q-1}}
\log\!\left(
1+\frac{M_o(2r)}{r^2\cC_{o,\rho}(r)}
\right)\,\dd r
=\infty,
\end{equation}
then every nonnegative global classical supersolution of
\eqref{eq:heat-ineq} vanishes identically.
\end{theorem}

\begin{remark}
For every measurable function
$R:[1,\infty)\to(0,\infty)$ satisfying $R(r)\ge2r$, the same proof gives
\begin{equation*}
\cH_{q,\sigma}(o)
\ge
c_q\int_1^\infty
\frac{r}{M_o(R(r))^{q-1}}
\log\!\left(
1+\frac{M_o(R(r))}{r^2\cC_{o,\rho}(r,R(r))}
\right)\,\dd r.
\end{equation*}
Thus the outer radius may be optimized scale by scale.
\end{remark}

Because of \eqref{eq:intro-annular-capacity-cutoff}, the logarithmic factor in
\eqref{eq:capacity-volume-condition} is at least $\log2$.  Hence Theorem~\ref{thm:capacity-volume-criterion} yields the following.

\begin{theorem}\label{thm:heat-energy-lower}
Let $q>1$.  There exists $c_q>0$, depending only on $q$, such that for every
$o\in V$,
\begin{equation}\label{eq:main-heat-lower}
        \cH_{q,\sigma}(o)
        \ge
        c_q\int_2^\infty
        \frac{r}{M_o(r)^{q-1}}\,\dd r.
\end{equation}
\end{theorem}

\begin{theorem}\label{thm:volume-criterion}
Let $q>1$.  Assume that, for some (and hence every) $o\in V$,
\begin{equation}\label{eq:volume-condition}
        \int_1^\infty
        \frac{r}{M_o(r)^{q-1}}\,\dd r
        =\infty.
\end{equation}
Then every nonnegative global classical supersolution of
\eqref{eq:heat-ineq} vanishes identically.
\end{theorem}

The finiteness or infiniteness of the integral in
\eqref{eq:volume-condition} is independent of the basepoint.  Indeed, for
$o,x\in V$,
\[
 B_\rho(x,r)\subset B_\rho(o,r+d_\rho(o,x)),
 \qquad
 B_\rho(o,r)\subset B_\rho(x,r+d_\rho(o,x)).
\]
The conclusion follows by a change of variables and the comparability of
$r$ and $r+d_\rho(o,x)$ for large $r$.

The next two corollaries concern the unit-speed, graph-distance case; the edge
conductances $\mu_{xy}$ remain arbitrary.

\begin{corollary}\label{cor:unweighted-volume-criterion}
If $\sigma\equiv1$ and $\rho\equiv1$, then \eqref{eq:volume-condition} is
equivalent to
\begin{equation}\label{eq:unweighted-volume-condition}
        \sum_{n=1}^\infty
        \frac{n}{\mu(B(o,n))^{q-1}}=\infty,
\end{equation}
where $B(o,n)$ is the graph-distance ball.  Under
\eqref{eq:unweighted-volume-condition}, every nonnegative global classical
supersolution of $\partial_tu-\Delta u\ge u^q$ vanishes identically.
\end{corollary}

\begin{corollary}\label{cor:unweighted-equation-volume-criterion}
Let $\sigma\equiv1$ and $\rho\equiv1$.  If
\eqref{eq:unweighted-volume-condition} holds, then the equation
\[
        \partial_tu-\Delta u=u^q
        \qquad\text{on }(0,\infty)\times V
\]
admits no positive global classical solution.  Moreover, for every nontrivial
nonnegative initial datum $u_0:V\to[0,\infty)$, the Cauchy problem
\[
\begin{cases}
        \partial_tu-\Delta u=u^q,
        &t>0,\ x\in V,\\
        u(0,x)=u_0(x),
        &x\in V,
\end{cases}
\]
admits no nonnegative global classical solution.
\end{corollary}

The second assertion is the standard Cauchy formulation of the Fujita
nonexistence conclusion.  The first assertion, and still more
Corollary~\ref{cor:unweighted-volume-criterion}, are formally stronger because
they do not require an initial trace.

We record a standard heat-kernel supersolution construction, which yields
small-data Cauchy solutions in the normalization used here; compare
\cite{Weissler81,LinWu17}.

\begin{proposition}\label{prop:heat-kernel-existence}
Let $q>1$ and $o\in V$.  If
\begin{equation}\label{eq:heat-kernel-existence-condition}
        \int_1^\infty
        \left(\sup_{x\in V}p_\sigma(t,o,x)\right)^{q-1}\,\dd t
        <\infty,
\end{equation}
then there exists $a_0>0$ such that, for every $a\in(0,a_0]$, the Cauchy
problem \eqref{eq:cauchy-problem} with
$u_0=a\1_{\{o\}}$ admits a positive global classical solution.
\end{proposition}

Notice that
\[
\sum_{x\in V}p_\sigma(t,o,x)^q\nu(x)
\le
\left(\sup_{x\in V}p_\sigma(t,o,x)\right)^{q-1}
\sum_{x\in V}p_\sigma(t,o,x)\nu(x)
\le
\left(\sup_{x\in V}p_\sigma(t,o,x)\right)^{q-1}.
\]
Together with $\sup_xp_\sigma(t,o,x)\le\nu(o)^{-1}$, this shows that
\eqref{eq:heat-kernel-existence-condition} implies
$\cH_{q,\sigma}(o)<\infty$, as is necessary in view of
Corollary~\ref{cor:heat-energy-liouville}.  Finiteness of
$\cH_{q,\sigma}(o)$ alone is not asserted to imply existence.

\begin{corollary}\label{cor:heat-kernel-volume-existence}
Let $q>1$.  Suppose that, for some $o\in V$ and $C>0$,
\begin{equation}\label{eq:anchored-heat-kernel-upper}
        \sup_{x\in V}p_\sigma(t,o,x)
        \le \frac{C}{M_o(\sqrt t)},
        \qquad t\ge1.
\end{equation}
If
\begin{equation}\label{eq:volume-convergence-condition}
        \int_1^\infty
        \frac{r}{M_o(r)^{q-1}}\,\dd r
        <\infty,
\end{equation}
then there exists $a_0>0$ such that, for every $a\in(0,a_0]$, the Cauchy
problem \eqref{eq:cauchy-problem} with
$u_0=a\1_{\{o\}}$ admits a positive global classical solution.
\end{corollary}

Thus, on graphs satisfying \eqref{eq:anchored-heat-kernel-upper},
Theorem~\ref{thm:volume-criterion} and
Corollary~\ref{cor:heat-kernel-volume-existence} give the exact
divergence--convergence alternative for the integral volume test.  In view of
Proposition~\ref{prop:supersolution-solution-equivalence}, this is an exact
alternative both for trace-free supersolutions and, at the existential level,
for the standard Cauchy problem: divergence excludes every nontrivial
nonnegative initial datum, whereas convergence produces positive global
solutions for all sufficiently small point-source data.  The heat-kernel
hypothesis is used only for this converse direction; convergence of the volume
integral alone is not asserted to imply existence on an arbitrary weighted
graph or for arbitrary small initial data.

The Euclidean Fujita phenomenon has an extensive parabolic literature.  After Fujita's
original work, nonexistence, grow-up, and critical-exponent results were
refined in \cite{Hayakawa73,KST77,Weissler80,Weissler81,BandleLevine89}; see
also \cite{Levine90,GalaktionovLevine98}.  On noncompact geometries and in the
presence of potentials, related phenomena were developed in
\cite{Pinsky97,Pinsky09,Zhang98,Zhang99,Zhang01,Ishige08,BandlePozioTesei11,MMP17,GXXS20}.
For heat-kernel estimates on graphs, see \cite{Delmotte,GT01}.  For
semilinear parabolic equations and inequalities on graphs, see
\cite{LinWu17,LinWu18,MPS26}.
The present proof is closest in spirit to the elliptic flow-decomposition
method of \cite{GHHS26}.

The proof proceeds in three steps.  First, a finite-cylinder test reduces
nonexistence to divergence of the heat-kernel $L^q$-energy
$\cH_{q,\sigma}(o)$.  Second, a Laplace--Hardy inequality reduces that heat
energy to nonlinear resolvent energies.  Third, the resolvent current is
decomposed into voltage-decreasing paths.  An ordinary edge is assigned its
intrinsic length $\rho$, while the killing edge at $x$ has conductance
$\lambda\nu(x)$ and length $\lambda^{-1/2}$.  The adaptedness condition
\eqref{eq:adapted-weight} together with the identity
$\lambda\nu(x)\lambda^{-1}=\nu(x)$ then unifies graph motion and killing within a single path-energy estimate.

The paper is organized as follows.  Section~2 establishes the relation among
supersolutions and solutions on the open cylinder, as well as Cauchy
solutions, and develops
finite-cylinder testing and the Laplace--Hardy reduction.  Section~3 proves
the resolvent--capacity estimates
by decomposing resolvent currents into paths.  Section~4 derives the parabolic
capacity estimate, the mixed capacity--volume criterion, and the integral
volume criterion.  Section~\ref{sec:sharpness} proves the
converse via a heat kernel condition and gives sharpness examples on integer lattices and logarithmically
perturbed weighted half-lines.

\section{Supersolutions, testing, and resolvent reduction}

Let $D\subset V$ be finite and connected.  The Dirichlet operator
$\Delta_{\sigma,D}$ is obtained by extending functions by zero outside $D$ and
then applying \eqref{eq:laplacian} on $D$.  Its heat kernel with respect to
$\nu$ is denoted by $p_{\sigma,D}(t,x,y)$.

\subsection{Supersolutions on the open cylinder and Cauchy solutions}

\begin{proposition}\label{prop:supersolution-solution-equivalence}
Let $q>1$.  The following are equivalent:
\begin{enumerate}[label=\textnormal{(\roman*)}]
\item \eqref{eq:heat-ineq} admits a nontrivial nonnegative global classical
supersolution on $(0,\infty)$;
\item \eqref{eq:heat-equation} admits a positive global classical solution on
$(0,\infty)$;
\item for some $o\in V$ and $a>0$, the Cauchy problem
\eqref{eq:cauchy-problem} with $u_0=a\1_{\{o\}}$ admits a positive global
classical solution.
\end{enumerate}
Equivalently, condition \textnormal{(i)} may be replaced by the existence of a
positive global classical supersolution.  More precisely, if $U$ is as in
\textnormal{(i)}, if $s>0$ and $o\in V$ satisfy $U(s,o)>0$, and if
$0<a<U(s,o)$, then the Cauchy problem with
$u_0=a\1_{\{o\}}$ has a positive global classical solution $u$ satisfying
\begin{equation}\label{eq:solution-below-supersolution}
        0<u(t,x)\le U(s+t,x),
        \qquad t>0,\ x\in V.
\end{equation}
\end{proposition}

\begin{proof}
The implications
\textnormal{(iii)}$\Rightarrow$\textnormal{(ii)}$\Rightarrow$\textnormal{(i)}
are immediate.  We prove \textnormal{(i)}$\Rightarrow$\textnormal{(iii)} and
the quantitative assertion.  Choose $s>0$ and $o\in V$ such that
$U(s,o)>0$, and fix any $a\in(0,U(s,o))$.  Put
\begin{equation*}
        f:=a\1_{\{o\}},
        \qquad
        \overline U(t,x):=U(s+t,x),
        \qquad t\ge0.
\end{equation*}
Choose an increasing exhaustion
$D_1\subset D_2\subset\cdots\uparrow V$ by finite connected sets containing
$o$.  For each $n$, let $u_n$ be the maximal nonnegative solution of the
finite-dimensional Dirichlet problem
\begin{equation}\label{eq:finite-nonlinear-dirichlet-problem}
\begin{cases}
\partial_tu_n-\Delta_{\sigma,D_n}u_n=u_n^q,
        &t>0,\ x\in D_n,\\
u_n(0,x)=f(x),
        &x\in D_n.
\end{cases}
\end{equation}
Applying standard ordinary differential equation theory to the locally
Lipschitz extension $r\mapsto(r_+)^q$, where $r_+:=\max\{r,0\}$, gives a
unique maximal solution.  The nonnegative cone is invariant, so this solution
satisfies \eqref{eq:finite-nonlinear-dirichlet-problem}.  We extend
$u_n(t,\cdot)$ by zero to $V$.

For $x\in D_n$, zero extension outside $D_n$ gives
\begin{align*}
(\partial_t-\Delta_{\sigma,D_n})\overline U(t,x)
&=(\partial_t-\Delta_\sigma)\overline U(t,x)
 +\frac1{\nu(x)}
  \sum_{\substack{y\notin D_n\\y\sim x}}
  \mu_{xy}\overline U(t,y)\\
&\ge \overline U(t,x)^q.
\end{align*}
Moreover, $f\le\overline U(0,\cdot)$ on $D_n$.  The comparison principle for
the cooperative finite-dimensional system therefore yields
\begin{equation}\label{eq:finite-solution-upper-bound}
        0\le u_n(t,x)\le\overline U(t,x)
\end{equation}
throughout the maximal interval of existence.  On every bounded time
interval the right-hand side is bounded on the finite set $D_n$; hence the
ordinary differential equation continuation criterion shows that $u_n$ is
global.

The solutions are monotone with respect to the domain.  Indeed, for
$x\in D_n$,
\begin{align*}
(\partial_t-\Delta_{\sigma,D_n})u_{n+1}(t,x)
&=(\partial_t-\Delta_{\sigma,D_{n+1}})u_{n+1}(t,x)\\
&\quad+
\frac1{\nu(x)}
\sum_{\substack{y\in D_{n+1}\setminus D_n\\y\sim x}}
\mu_{xy}u_{n+1}(t,y)\\
&\ge u_{n+1}(t,x)^q.
\end{align*}
Thus $u_{n+1}|_{D_n}$ is a supersolution of
\eqref{eq:finite-nonlinear-dirichlet-problem} with the same initial value as
$u_n$.  A second application of comparison gives
\begin{equation*}
        u_n(t,x)\le u_{n+1}(t,x),
        \qquad t\ge0,\ x\in D_n.
\end{equation*}
Consequently the pointwise limit
\begin{equation*}
        u(t,x):=\lim_{n\to\infty}u_n(t,x)
\end{equation*}
exists and, by \eqref{eq:finite-solution-upper-bound}, satisfies
$0\le u(t,x)\le\overline U(t,x)$.

We next pass to the equation.  Fix $x\in V$.  By local finiteness, for all
sufficiently large $n$ the set $D_n$ contains $x$ and all its neighbors.  For
such $n$,
\begin{equation}\label{eq:finite-solution-integral-equation}
\begin{split}
u_n(t,x)=f(x)+\int_0^t\biggl[
&\frac1{\nu(x)}\sum_{y\sim x}\mu_{xy}
       \bigl(u_n(r,y)-u_n(r,x)\bigr)\\
&+u_n(r,x)^q\biggr]~\dd r.
\end{split}
\end{equation}
On a fixed compact time interval, the absolute value of the integrand is
bounded independently of $n$ by
\begin{equation*}
\frac1{\nu(x)}\sum_{y\sim x}\mu_{xy}
 \bigl(\overline U(r,y)+\overline U(r,x)\bigr)
 +\overline U(r,x)^q,
\end{equation*}
which is integrable because $x$ has only finitely many neighbors.  Dominated
convergence in \eqref{eq:finite-solution-integral-equation} gives
\begin{equation}\label{eq:limiting-solution-integral-equation}
        u(t,x)=f(x)+\int_0^t
        \bigl(\Delta_\sigma u(r,x)+u(r,x)^q\bigr)~\dd r.
\end{equation}
It follows first that $u(\cdot,x)$ is continuous for every $x$.  The
integrand in \eqref{eq:limiting-solution-integral-equation} is then continuous,
again by local finiteness, so $u(\cdot,x)$ is $C^1$ and satisfies
\eqref{eq:heat-equation} pointwise.

It remains to prove strict positivity.  The variation-of-constants formula on
$D_n$ gives
\begin{equation*}
        u_n(t,\cdot)
        =e^{t\Delta_{\sigma,D_n}}f
        +\int_0^t e^{(t-r)\Delta_{\sigma,D_n}}
                u_n(r,\cdot)^q~\dd r
        \ge e^{t\Delta_{\sigma,D_n}}f.
\end{equation*}
The Dirichlet semigroup on a finite connected set is positivity improving,
since its generator matrix has nonnegative off-diagonal entries and is
irreducible.  For any fixed $x\in V$, choose $n$ with $x\in D_n$.  Since
$o\in D_n$ and
$f(o)=a>0$, it follows that
\begin{equation*}
        u(t,x)\ge u_n(t,x)>0,
        \qquad t>0.
\end{equation*}
Equation~\eqref{eq:limiting-solution-integral-equation} also gives
$u(0,\cdot)=f$ and pointwise continuity at $t=0$.  Together with
$u\le\overline U$, this proves \eqref{eq:solution-below-supersolution}, the
Cauchy assertion, and hence the proposition.
\end{proof}

\begin{remark}
The proposition separates the two time-domain issues.  Its hypothesis is
trace-free, but the solution produced from it is a genuine global Cauchy
solution.  Consequently, a nontrivial global supersolution on $(0,\infty)\times V$
exists if and only if some nontrivial nonnegative datum admits a global Cauchy
solution; one may always take a sufficiently small point source.  More generally, if a nonnegative datum $f$ satisfies
$f\le U(s,\cdot)$ for some $s>0$, the same exhaustion argument constructs a
nonnegative global classical Cauchy solution with initial value $f$ and with
$u(t,x)\le U(s+t,x)$.  Thus, for a prescribed initial datum, the relevant
extra hypothesis is domination by a supersolution at the initial time.
\end{remark}

\subsection{Finite-cylinder testing}

For $T>0$ and $o\in D$, set
\[
        \cH_{D,T}^\sigma(o)
        :=
        \int_0^T\sum_{x\in D}p_{\sigma,D}(t,o,x)^q\nu(x)\,\dd t.
\]

\begin{lemma}\label{lem:testing}
Let $q>1$.  If $u$ is a nonnegative global classical supersolution of
\eqref{eq:heat-ineq}, then for every finite connected $D\ni o$, every $T>0$,
and every time shift $s>0$,
\begin{equation}\label{eq:testing-bound}
        u(s,o)\nu(o)
        \le
        \left(\frac{q}{q-1}\right)^{q/(q-1)}
        \bigl(\cH_{D,T}^\sigma(o)\bigr)^{-1/(q-1)}.
\end{equation}
\end{lemma}

\begin{proof}
Write
\[
        K(t,x):=p_{\sigma,D}(t,o,x),
        \qquad 0<t<T,
\]
and set $H:=\cH_{D,T}^\sigma(o)$.  Since $D$ is finite and $T<\infty$,
$0<H<\infty$.  Define
\[
        \Psi(t,x)
        :=
        H^{-1}
        \int_t^T\sum_{y\in D}
        p_{\sigma,D}(r-t,x,y)K(r,y)^{q-1}\nu(y)\,\dd r.
\]
Then
\[
        (-\partial_t-\Delta_{\sigma,D})\Psi
        =H^{-1}K^{q-1},
        \qquad
        \Psi(T,\cdot)=0,
\]
and the semigroup property gives
\begin{equation}\label{eq:psi-normalization}
        \Psi(0,o)
        =
        H^{-1}
        \int_0^T\sum_{y\in D}p_{\sigma,D}(r,o,y)^q\nu(y)\,\dd r
        =1.
\end{equation}
Set
\[
        q':=\frac{q}{q-1},
        \qquad
        \phi:=\Psi^{q'}.
\]
Since $a\mapsto a^{q'}$ is convex and $\Delta_{\sigma,D}$ has nonnegative
off-diagonal coefficients,
\[
        \Delta_{\sigma,D}(\Psi^{q'})(x)
        \ge q'\Psi(x)^{q'-1}\Delta_{\sigma,D}\Psi(x).
\]
Consequently,
\begin{equation}\label{eq:chain-ineq}
        (-\partial_t-\Delta_{\sigma,D})\phi
        \le
        \frac{q'}{H}
        \Psi^{1/(q-1)}K^{q-1}.
\end{equation}

Let $v(t,x):=u(s+t,x)$.  Multiplying
$\partial_tv-\Delta_\sigma v\ge v^q$ by $\phi$ and summing over $D$ uses the
identity
\[
\begin{aligned}
\sum_{x\in D}(-\Delta_\sigma v)(x)\phi(x)\nu(x)
={}&
\sum_{x\in D}v(x)(-\Delta_{\sigma,D}\phi)(x)\nu(x)\\
&-
\sum_{\substack{x\in D,\ y\notin D\\x\sim y}}
\mu_{xy}v(y)\phi(x).
\end{aligned}
\]
The last term is nonpositive.  Integrating over $(0,T)$ and using
$\phi(T,\cdot)=0$ therefore gives
\begin{equation}\label{eq:testing-ibp}
\begin{split}
        \int_0^T\sum_{x\in D} v(t,x)^q\phi(t,x)\nu(x)\,\dd t
        &+\sum_{x\in D} v(0,x)\phi(0,x)\nu(x)\\
        &\le
        \int_0^T\sum_{x\in D}
        v(t,x)(-\partial_t-\Delta_{\sigma,D})\phi(t,x)\nu(x)\,\dd t.
\end{split}
\end{equation}
Combining \eqref{eq:chain-ineq} with H\"older's inequality yields
\begin{align*}
\text{right-hand side of }\eqref{eq:testing-ibp}
&\le
\frac{q'}{H}
\int_0^T\sum_{x\in D}
v\Psi^{1/(q-1)}K^{q-1}\nu\,\dd t\\
&\le
q'H^{-1/q}
\left(
\int_0^T\sum_{x\in D}v^q\Psi^{q'}\nu\,\dd t
\right)^{1/q}.
\end{align*}
If $A$ denotes the first term on the left of
\eqref{eq:testing-ibp} and $B$ the second, then
\[
        A+B\le q'H^{-1/q}A^{1/q}
        \le q'H^{-1/q}(A+B)^{1/q}.
\]
It follows that
\[
        A+B\le(q')^{q'}H^{-1/(q-1)}.
\]
Using \eqref{eq:psi-normalization}, $\phi(0,o)=1$, and retaining only the
$x=o$ term proves \eqref{eq:testing-bound}.
\end{proof}


\begin{corollary}\label{cor:heat-energy-liouville}
If, for some $o\in V$,
\[
        \int_0^\infty\sum_xp_\sigma(t,o,x)^q\nu(x)\,\dd t=\infty,
\]
then every nonnegative global classical supersolution of \eqref{eq:heat-ineq}
is identically zero.
\end{corollary}

\begin{proof}
Let $D_j\uparrow V$ be an exhaustion by finite connected sets containing $o$.
Extend each killed kernel $p_{\sigma,D_j}(t,o,\cdot)$ by zero outside $D_j$.
These kernels increase pointwise to the minimal kernel
$p_\sigma(t,o,\cdot)$.  Hence
\[
        \lim_{j\to\infty}\lim_{T\to\infty}
        \cH_{D_j,T}^\sigma(o)=\infty.
\]
Lemma~\ref{lem:testing} gives $u(s,o)=0$ for every $s>0$.  Therefore
$\partial_tu(s,o)=0$, and \eqref{eq:heat-ineq} implies
$-\Delta_\sigma u(s,o)\ge0$.  On the other hand,
\[
        \Delta_\sigma u(s,o)
        =\frac1{\nu(o)}\sum_{y\sim o}\mu_{oy}u(s,y)\ge0.
\]
Thus every neighbor of $o$ vanishes at every positive time.  Connectedness
then gives $u\equiv0$.
\end{proof}

\subsection{Laplace--Hardy reduction}

For $\lambda>0$, define
\[
        G_\lambda^\sigma(o,x)
        :=\int_0^\infty e^{-\lambda t}p_\sigma(t,o,x)\,\dd t,
\]
and
\[
        L_\lambda^\nu(o)
        :=\sum_xG_\lambda^\sigma(o,x)^q\nu(x).
\]

\begin{lemma}\label{lem:laplace-hardy}
Let $q>1$.  There exists $c_q>0$ such that for every nonnegative
$f\in L^q(0,\infty)$,
\begin{equation*}
        \int_0^\infty f(t)^q\,\dd t
        \ge
        c_q\int_0^\infty\lambda^{q-2}
        \left(\int_0^\infty e^{-\lambda t}f(t)\,\dd t\right)^q
        \,\dd\lambda.
\end{equation*}
\end{lemma}

\begin{proof}
After the change of variables $\lambda=s^{-1}$, the right-hand side, without the factor $c_q$,
equals
\[
        \int_0^\infty
        \left(\frac1s\int_0^\infty e^{-t/s}f(t)\,\dd t\right)^q\,\dd s.
\]
It is enough to show that
\[
        Tf(s):=\frac1s\int_0^\infty e^{-t/s}f(t)\,\dd t
\]
is bounded on $L^q(0,\infty)$.  Decompose dyadically:
\[
Tf(s)
\le
\frac1s\int_0^sf(t)\,\dd t
+\sum_{k=0}^\infty e^{-2^k}\frac1s
\int_{2^ks}^{2^{k+1}s}f(t)\,\dd t.
\]
The first term is the Hardy averaging operator.  The $k$-th remaining term is
denoted by
\[
        B_kf(s):=\frac1s\int_{2^ks}^{2^{k+1}s}f(t)\,\dd t.
\]
H\"older's inequality and Tonelli's theorem give
\[
        \|B_kf\|_{L^q}^q
        \le 2^{k(q-1)}\log2\,\|f\|_{L^q}^q.
\]
Thus the $k$-th summand is bounded in $L^q$ by
$e^{-2^k}2^{k(1-1/q)}(\log2)^{1/q}\|f\|_{L^q}$.  Hardy's
inequality and summability in $k$ complete the proof.
\end{proof}

Apply Lemma~\ref{lem:laplace-hardy} to
$f_x(t)=p_\sigma(t,o,x)\1_{(0,T)}(t)$.  Sub-Markovianity and symmetry give
$p_\sigma(t,o,x)\le\nu(x)^{-1}$, so $f_x\in L^q(0,\infty)$.  Sum the resulting
inequalities with respect to $\nu$, use Tonelli's theorem, and then let
$T\to\infty$.  Monotone convergence on both sides gives:

\begin{corollary}\label{cor:resolvent-reduction}
For every $q>1$, there exists $c_q>0$ such that
\[
        \cH_{q,\sigma}(o)
        \ge
        c_q\int_0^\infty\lambda^{q-2}L_\lambda^\nu(o)\,\dd\lambda.
\]
\end{corollary}

\section{Flow decomposition and resolvent-capacity estimates}

\subsection{The resolvent network with metric killing edges}

We next prove a resolvent-capacity lower bound for $L_\lambda^\nu(o)$.  The
argument is first carried out on a finite set and then passed to the limit.
The main point is to incorporate both the intrinsic geometry and the killing
term into one path metric.  Every retained graph edge keeps its intrinsic
length $\rho$, while a killing edge is given length $\lambda^{-1/2}$.  With
this choice every current path has enough length on the natural resolvent
scale, including paths killed before they leave an intrinsic ball.

Let $D\subset V$ be finite and connected with $o\in D$, and let
\[
        g(x):=G_{\lambda,D}^\sigma(o,x)
        =\int_0^\infty e^{-\lambda t}p_{\sigma,D}(t,o,x)\,\dd t,
        \qquad x\in D.
\]
Extend $g$ by zero outside $D$.  Then
\begin{equation}\label{eq:finite-resolvent-equation}
        \sum_{y\sim x}\mu_{xy}\bigl(g(x)-g(y)\bigr)
        +\lambda\nu(x)g(x)
        =\1_{\{o\}}(x),
        \qquad x\in D.
\end{equation}

We realize \eqref{eq:finite-resolvent-equation} on a finite augmented network.
For every boundary edge $e=\{x,y\}$ with $x\in D$ and $y\notin D$, introduce
a terminal copy $b_e$ and replace $e$ by an edge $\{x,b_e\}$ of conductance
$\mu_{xy}$.  For every $x\in D$, introduce a killing terminal $k_x$ and an
edge $\{x,k_x\}$ of conductance $\lambda\nu(x)$.  Let $\mathcal T_D$ denote
the collection of all boundary and killing terminals.  Extend $g$ to the
augmented network by setting $g=0$ on $\mathcal T_D$.  The finite-domain
strong maximum principle gives
\[
        g(x)>0,\qquad x\in D.
\]
A retained graph edge $e=\{x,y\}$ and its copied boundary edge are assigned
length $\rho_e:=\rho(x,y)$, whereas every killing edge is assigned length
\[
        \ell_\lambda:=\lambda^{-1/2}.
\]
Define a radial label on the augmented vertices by
\[
\mathfrak r(z)=
\begin{cases}
 d_\rho(o,z),&z\in D,\\
 d_\rho(o,y),&z=b_e,\ e=\{x,y\},\\
 d_\rho(o,x)+\lambda^{-1/2},&z=k_x.
\end{cases}
\]
The path-metric inequality gives
$|d_\rho(o,x)-d_\rho(o,y)|\le d_\rho(x,y)\le\rho(x,y)$ on every graph
edge.  Thus the radial increment across every retained or copied edge is
bounded by its assigned length, and the same is true by definition for a
killing edge.

Orient every edge with nonzero voltage drop from the larger value of $g$ to
the smaller one and set
\[
        \theta_e=c_e\bigl(g(e^-)-g(e^+)\bigr),
\]
where $c_e$ is its conductance.  Equation
\eqref{eq:finite-resolvent-equation} says that $\theta$ is a unit acyclic flow
from $o$ to $\mathcal T_D$.  The standard flow-decomposition theorem yields a
finite atomic probability measure on directed source-to-terminal paths; see,
for example, \cite{FordFulkerson62}.  Hence there is a probability measure
$\Prob$ on directed voltage-decreasing paths
\[
        \gamma=(x_0=o,e_0,x_1,\ldots,e_{m-1},x_m),
        \qquad x_m\in\mathcal T_D,
\]
such that
\begin{equation}\label{eq:path-decomposition}
        \Prob(\gamma\text{ uses }e)=\theta_e
        \qquad\text{for every directed edge }e.
\end{equation}

For a sampled path write
\[
        V_i:=g(x_i),\qquad
        \delta_i:=V_i-V_{i+1}>0,\qquad
        \ell_i:=\text{length of }e_i,
\]
and
\[
        s_i:=\sum_{j=0}^{i-1}\ell_j,
        \qquad
        L_\gamma:=s_m.
\]
Here $V_m=0$.  Since the radial increment along an edge is at most its length,
\begin{equation}\label{eq:radial-prefix-bound}
        \mathfrak r(x_i)\le s_i,
        \qquad 0\le i\le m.
\end{equation}
In particular, if $D\supset B_\rho(o,R)$ and $R<\lambda^{-1/2}$, then every
terminal has radial label larger than $R$, and hence
\begin{equation}\label{eq:path-length-resolvent-scale}
        L_\gamma>R.
\end{equation}

Define the path energy
\[
        H_\gamma:=\sum_{i=0}^{m-1}
        \frac{\ell_i^2V_i^q}{\delta_i}.
\]

\begin{lemma}\label{lem:path-energy-representation}
For every $\lambda>0$,
\begin{equation*}
        \E H_\gamma
        \le 2\sum_{x\in D}g(x)^q\nu(x).
\end{equation*}
\end{lemma}

\begin{proof}
Using \eqref{eq:path-decomposition},
\[
\E H_\gamma
=
\sum_e\theta_e\frac{\ell_e^2g(e^-)^q}{g(e^-)-g(e^+)}
=
\sum_e c_e\ell_e^2g(e^-)^q.
\]
At a vertex $x\in D$, the total contribution of retained and copied
graph edges is at most
\[
        g(x)^q\sum_{y\sim x}\mu_{xy}\rho(x,y)^2
        \le \nu(x)g(x)^q
\]
by \eqref{eq:adapted-weight}.  The killing edge contributes
\[
        \lambda\nu(x)\lambda^{-1}g(x)^q=\nu(x)g(x)^q.
\]
Summing over $x$ proves the claim.
\end{proof}

\subsection{Hardy inequality along the current paths}

Let $v_\gamma:[0,L_\gamma]\to[0,\infty)$ be the piecewise affine function
satisfying $v_\gamma(s_i)=V_i$.

\begin{lemma}\label{lem:path-hardy}
Let $q>1$.  If $v:[0,L]\to[0,\infty)$ is continuous, strictly decreasing,
and affine on each interval of a finite partition, then
\begin{equation}\label{eq:path-hardy}
        \int_0^L\frac{v(s)^q}{-v'(s)}\,\dd s
        \ge
        \frac{q-1}{4\log2}
        \int_0^L s\,v(s)^{q-1}\,\dd s.
\end{equation}
\end{lemma}

\begin{proof}
Set
\[
        a(s):=\frac{-v'(s)}{v(s)^q},
        \qquad
        A(t):=\int_0^t a(s)\,\dd s.
\]
For $0<t<L$, Cauchy's inequality gives
\[
        \int_{t/2}^t\frac{\dd s}{a(s)}
        \ge \frac{t^2}{4A(t)}.
\]
Dividing by $t$, integrating in $t$, and changing the order of
integration,
\[
        \int_0^L\frac{t}{A(t)}\,\dd t
        \le 4\log2\int_0^L\frac{\dd s}{a(s)}.
\]
Moreover,
\[
        A(t)
        =\int_0^t\frac{-v'(s)}{v(s)^q}\,\dd s
        \le \frac{v(t)^{1-q}}{q-1}.
\]
Thus $A(t)^{-1}\ge(q-1)v(t)^{q-1}$, which proves
\eqref{eq:path-hardy}.
\end{proof}

On $(s_i,s_{i+1})$ one has $-v_\gamma'=\delta_i/\ell_i$ and
$v_\gamma\le V_i$.  Therefore Lemma~\ref{lem:path-hardy} gives
\begin{equation}\label{eq:path-energy-controls-hardy}
        H_\gamma
        \ge
        \int_0^{L_\gamma}\frac{v_\gamma(s)^q}{-v_\gamma'(s)}\,\dd s
        \ge
        c_q\int_0^{L_\gamma}s\,v_\gamma(s)^{q-1}\,\dd s.
\end{equation}

For $0<r<R$, where $D\supset B_\rho(o,R)$ and $R<\lambda^{-1/2}$, define
\[
        \kappa_r
        :=\min\{0\le i\le m-1:\mathfrak r(x_{i+1})>r\},
        \qquad
        Z_r(\gamma):=V_{\kappa_r}.
\]
The index is well-defined because $\mathfrak r(x_m)>R>r$.  By minimality,
\begin{equation}\label{eq:first-exit-location}
        \mathfrak r(x_{\kappa_r})\le r
        <\mathfrak r(x_{\kappa_r+1}).
\end{equation}
Thus $Z_r$ is the voltage at the tail of the first edge by which the augmented
path leaves the radial ball of radius $r$.

\begin{lemma}\label{lem:radial-first-exit-hardy}
There exists $c_q>0$, depending only on $q$, such that
\begin{equation*}
        H_\gamma
        \ge c_q\int_0^R r\,Z_r(\gamma)^{q-1}\,\dd r.
\end{equation*}
\end{lemma}

\begin{proof}
By \eqref{eq:path-length-resolvent-scale}, one has $R<L_\gamma$.  Put
\[
        J_\gamma:=\int_0^{L_\gamma}s\,v_\gamma(s)^{q-1}\,\dd s.
\]
By \eqref{eq:path-energy-controls-hardy}, $H_\gamma\ge c_qJ_\gamma$.
Split the radii into
\[
E_1:=\{0<r<R:s_{\kappa_r}\ge r/2\},
\qquad
E_2:=\{0<r<R:s_{\kappa_r}<r/2\}.
\]
If $r\in E_1$, monotonicity gives
$Z_r=v_\gamma(s_{\kappa_r})\le v_\gamma(r/2)$, and hence
\[
        \int_{E_1}rZ_r^{q-1}\,\dd r
        \le 4\int_0^{R/2}s\,v_\gamma(s)^{q-1}\,\dd s
        \le4J_\gamma.
\]
If $r\in E_2$ and $\kappa_r=i$, then
\eqref{eq:radial-prefix-bound} and \eqref{eq:first-exit-location} imply
$r<2\ell_i$.  Writing $E_{2,i}:=\{r\in E_2:\kappa_r=i\}$ and using
$\delta_i\le V_i$, we obtain
\[
        \int_{E_{2,i}}rZ_r^{q-1}\,\dd r
        \le 2\ell_i^2V_i^{q-1}
        \le 2\frac{\ell_i^2V_i^q}{\delta_i}.
\]
Summing over $i$ and using $J_\gamma\le C_qH_\gamma$ proves the result.
\end{proof}

\subsection{Stopped flows and resolvent capacities}

We use the following stopped-flow estimate in the augmented network.  For
disjoint nonempty vertex sets $A$ and $K$ in a finite electrical network, we
write $\Ceff(A,K)$ for the effective conductance, normalized as the minimum
Dirichlet energy among potentials equal to $1$ on $A$ and $0$ on $K$.

\begin{lemma}\label{lem:reciprocal-stopped-flow}
Suppose that the network carries an acyclic unit current $\theta=i_g$
generated by a nonnegative voltage $g$, with path decomposition $\Prob$.  On
each sampled path choose a segment beginning at a vertex in $A$ and ending at
a later vertex in $K$, and let $\tau<\zeta$ denote its initial and terminal
indices.  Then
\begin{equation*}
        \E\frac1{g(x_\tau)}
        \le \Ceff(A,K).
\end{equation*}
\end{lemma}

\begin{proof}
Let $\alpha=\alpha(\gamma)\ge0$ be an arbitrary path multiplier.  For a
directed edge $e$, set
\[
        S_e:=\{\gamma:e\text{ belongs to the selected segment }
        \gamma[\tau,\zeta]\}
\]
and define
\[
        \Phi_\alpha(e):=\int_{S_e}\alpha(\gamma)\,\dd\Prob(\gamma).
\]
Thus $\Phi_\alpha$ is the superposition of the selected unit path flows, with
the flow carried by $\gamma$ multiplied by $\alpha(\gamma)$.  It is a flow
from $A$ to $K$ of strength
\[
        I_\alpha:=\E\alpha.
\]
Thomson's principle therefore gives
\begin{equation*}
        \frac{I_\alpha^2}{\Ceff(A,K)}\le\cE(\Phi_\alpha).
\end{equation*}

If $e$ has voltage drop $\delta_e=g(e^-)-g(e^+)$, then
$\theta_e=c_e\delta_e$.  Moreover,
$S_e\subset\{\gamma:\gamma\text{ uses }e\}$, and hence
$\Prob(S_e)\le\theta_e$ by the path-decomposition identity.  Cauchy's
inequality now yields
\[
        \Phi_\alpha(e)^2
        \le
        \theta_e\int_{S_e}\alpha(\gamma)^2\,\dd\Prob(\gamma).
\]
Consequently, Tonelli's theorem and telescoping of the voltage drops along each
selected segment give
\begin{align*}
        \cE(\Phi_\alpha)
        &=\sum_e\frac{\Phi_\alpha(e)^2}{c_e}\\
        &\le
        \E\left[
        \alpha(\gamma)^2\bigl(g(x_\tau)-g(x_\zeta)\bigr)
        \right]\\
        &\le
        \E\left[\alpha(\gamma)^2g(x_\tau)\right],
\end{align*}
since $g\ge0$.  We have thus proved, for every nonzero nonnegative $\alpha$, that
\begin{equation*}
        \frac{(\E\alpha)^2}
        {\E\left[\alpha^2g(x_\tau)\right]}
        \le \Ceff(A,K).
\end{equation*}

Set $X(\gamma):=g(x_\tau)$.  Since the selected segment follows a strictly
voltage-decreasing path and $\tau<\zeta$, one has $X>0$.  By Cauchy's
inequality,
\[
        (\E\alpha)^2
        =\left(\E\left[\alpha X^{1/2}X^{-1/2}\right]\right)^2
        \le \E(\alpha^2X)\,\E(X^{-1}),
\]
and equality holds when $\alpha$ is proportional to $X^{-1}$.  Thus the
natural optimizing choice is
\[
        \alpha(\gamma)=\frac1{g(x_\tau)}.
\]
Substituting this choice into the preceding flow inequality gives
\[
        \frac{\left(\E X^{-1}\right)^2}{\Ceff(A,K)}
        \le \E X^{-1},
\]
which proves the claim.
\end{proof}

For a nonempty set $A\subset D$, define its $\lambda$-capacity relative to domain $D$ by
\begin{equation}\label{eq:finite-lambda-capacity-definition}
\Cap_{\lambda,D}^\nu(A)
:=
\inf\left\{
\mathcal E_{\mathrm D}(f)+\lambda\sum_{x\in D}f(x)^2\nu(x):
 f\in\ell_0(V),\quad
 f=1\text{ on }A,\quad f=0\text{ on }V\setminus D
\right\}.
\end{equation}
After putting every terminal at voltage zero, the augmented-network energy is
exactly the functional in \eqref{eq:finite-lambda-capacity-definition}.
Dirichlet's principle therefore gives the identification
\begin{equation}\label{eq:finite-lambda-capacity-identification}
\Cap_{\lambda,D}^\nu(A)
=
\Ceff^{D,\lambda}(A,\mathcal T_D).
\end{equation}

For $0<r<R$, apply Lemma~\ref{lem:reciprocal-stopped-flow} with
$A=B_\rho(o,r)$, $K=\mathcal T_D$, $\tau=\kappa_r$, and the terminal time as the
endpoint of the selected segment.  By
\eqref{eq:finite-lambda-capacity-identification},
\begin{equation*}
        \E Z_r^{-1}
        \le \Cap_{\lambda,D}^\nu\bigl(B_\rho(o,r)\bigr).
\end{equation*}
H\"older's inequality gives
\[
        1
        \le
        \bigl(\E Z_r^{q-1}\bigr)^{1/q}
        \bigl(\E Z_r^{-1}\bigr)^{(q-1)/q}.
\]
Therefore
\begin{equation}\label{eq:resolvent-capacity-positive-moment}
        \E Z_r^{q-1}
        \ge
        \Cap_{\lambda,D}^\nu\bigl(B_\rho(o,r)\bigr)^{1-q}.
\end{equation}

The preceding estimates give the finite-domain capacitary inequality directly.

\begin{proposition}\label{prop:finite-resolvent-capacity-lower}
Let $q>1$, $\lambda>0$, and suppose that $D\supset B_\rho(o,R)$ with
$0<R<\lambda^{-1/2}$.  Then
\begin{equation}\label{eq:finite-resolvent-capacity-lower}
L_{\lambda,D}^\nu(o)
:=\sum_{x\in D}G_{\lambda,D}^\sigma(o,x)^q\nu(x)
\ge
c_q\int_0^R
r\,\Cap_{\lambda,D}^\nu\bigl(B_\rho(o,r)\bigr)^{1-q}\,\dd r.
\end{equation}
\end{proposition}

\begin{proof}
Lemmas~\ref{lem:path-energy-representation} and~\ref{lem:radial-first-exit-hardy}, followed by Tonelli's theorem, give
\begin{align*}
L_{\lambda,D}^\nu(o)
&\ge c_q\E H_\gamma\\
&\ge c_q\int_0^R r\,\E Z_r^{q-1}\,\dd r.
\end{align*}
Substituting \eqref{eq:resolvent-capacity-positive-moment} proves
\eqref{eq:finite-resolvent-capacity-lower}.
\end{proof}

We next pass from the killed capacities to the whole-graph capacity
\eqref{eq:lambda-capacity-definition}.

\begin{lemma}\label{lem:lambda-capacity-exhaustion}
Let $A\subset V$ be finite and nonempty, let $\lambda>0$, and let
$D_j\uparrow V$ be an exhaustion by finite connected sets containing $A$.
Then
\begin{equation*}
\Cap_{\lambda,D_j}^\nu(A)\downarrow\Cap_\lambda^\nu(A).
\end{equation*}
\end{lemma}

\begin{proof}
If $D_j\subset D_{j+1}$, every function admissible for
$\Cap_{\lambda,D_j}^\nu(A)$ is admissible for
$\Cap_{\lambda,D_{j+1}}^\nu(A)$, so the finite-domain capacities decrease.  Every
such function is finitely supported and admissible for $\Cap_\lambda^\nu(A)$;
hence their limit is at least $\Cap_\lambda^\nu(A)$.

Conversely, let $\varepsilon>0$.  Choose $f\in\ell_0(V)$, with $f\ge1$ on
$A$, such that
\[
\mathcal E_{\mathrm D}(f)+\lambda\sum_x f(x)^2\nu(x)
\le \Cap_\lambda^\nu(A)+\varepsilon.
\]
After truncation, assume $f=1$ on $A$ and $0\le f\le1$.  For all sufficiently
large $j$, $\operatorname{supp}f\subset D_j$, and hence
\[
\Cap_{\lambda,D_j}^\nu(A)
\le \mathcal E_{\mathrm D}(f)+\lambda\sum_x f(x)^2\nu(x)
\le \Cap_\lambda^\nu(A)+\varepsilon.
\]
Letting $\varepsilon\downarrow0$ proves the claim.
\end{proof}

\begin{proposition}\label{prop:resolvent-capacity-lower}
Let $q>1$ and fix $\eta\in(0,1)$.  Then, for every
$\lambda>0$,
\begin{equation}\label{eq:resolvent-capacity-lower}
L_\lambda^\nu(o)
\ge
c_q\int_0^{\eta\lambda^{-1/2}}
r\,\Cap_\lambda^\nu\bigl(B_\rho(o,r)\bigr)^{1-q}\,\dd r.
\end{equation}
\end{proposition}

\begin{proof}
Put $R=\eta\lambda^{-1/2}$ and choose an exhaustion $D_j\uparrow V$ with
$D_j\supset B_\rho(o,R)$.  Proposition~\ref{prop:finite-resolvent-capacity-lower}
gives
\[
L_{\lambda,D_j}^\nu(o)
\ge
c_q\int_0^R
r\,\Cap_{\lambda,D_j}^\nu\bigl(B_\rho(o,r)\bigr)^{1-q}\,\dd r.
\]
The Dirichlet resolvents increase pointwise to $G_\lambda^\sigma(o,\cdot)$, so
$L_{\lambda,D_j}^\nu(o)\uparrow L_\lambda^\nu(o)$.  By
Lemma~\ref{lem:lambda-capacity-exhaustion} and $1-q<0$, the integrands on the
right increase pointwise to
$r\Cap_\lambda^\nu(B_\rho(o,r))^{1-q}$.  Monotone convergence proves
\eqref{eq:resolvent-capacity-lower}.
\end{proof}

\begin{remark}
Proposition~\ref{prop:resolvent-capacity-lower} contains the elliptic
Green-energy estimate as the limit $\lambda\downarrow0$.  For a nonempty
finite set $A\subset V$, define its capacity to infinity by
\[
\Cap_V(A):=
\inf\left\{\mathcal E_{\mathrm D}(f):
 f\in\ell_0(V),\quad f\ge1\text{ on }A\right\}.
\]
Letting $\lambda\downarrow0$ first in the finite-domain resolvent equation
\eqref{eq:finite-resolvent-equation} and then along an exhaustion shows that
\[
G_\lambda^\sigma(o,x)\uparrow
 g(o,x):=\int_0^\infty p_\sigma(t,o,x)\,\dd t\in(0,\infty],
\]
where $g$ is the extended whole-graph Green kernel of $-\Delta$ and is
independent of the speed weight $\sigma$.  When finite, it satisfies
$-\Delta g(o,\cdot)=\mu(o)^{-1}\1_{\{o\}}$.  The variational definitions also
give
\[
        \Cap_\lambda^\nu(A)\downarrow\Cap_V(A).
\]
Indeed, the lower bound by $\Cap_V(A)$ is immediate, whereas a finitely
supported near-minimizer for $\Cap_V(A)$ gives the reverse bound as
$\lambda\downarrow0$.
Since $\eta\lambda^{-1/2}\uparrow\infty$, monotone convergence in
\eqref{eq:resolvent-capacity-lower} yields, with the convention
$0^{1-q}=+\infty$,
\[
\sum_{x\in V}g(o,x)^q\nu(x)
\ge
c_q\int_0^\infty
r\,\Cap_V\bigl(B_\rho(o,r)\bigr)^{1-q}\,\dd r.
\]
This is exactly the elliptic capacity estimate
\cite[Proposition~6.13]{GHHS26}.
\end{remark}

The $\lambda$-capacity can be estimated more sharply by retaining separately
the annular energy and the $L^2$ mass.

\begin{lemma}\label{lem:massive-annular-capacity-comparison}
For every $\lambda>0$ and $0<r<R$,
\begin{equation*}
\Cap_\lambda^\nu\bigl(B_\rho(o,r)\bigr)
\le
\cC_{o,\rho}(r,R)+\lambda M_o(R).
\end{equation*}
In particular,
\begin{equation}\label{eq:annular-capacity-cutoff}
        \cC_{o,\rho}(r)=\cC_{o,\rho}(r,2r)\le r^{-2}M_o(2r).
\end{equation}
More generally, if $R\ge2r$, then
\begin{equation*}
        \cC_{o,\rho}(r,R)\le r^{-2}M_o(R).
\end{equation*}
\end{lemma}

\begin{proof}
Since $B_\rho(o,R)$ is finite, the variational problem
\eqref{eq:annular-capacity-definition} has an equilibrium potential
$h_{r,R}$.  Truncation gives
$0\le h_{r,R}\le1$, while $h_{r,R}=1$ on $B_\rho(o,r)$ and
$h_{r,R}=0$ outside $B_\rho(o,R)$.  Hence $h_{r,R}$ is admissible for the $\lambda$-capacity and
\[
\Cap_\lambda^\nu\bigl(B_\rho(o,r)\bigr)
\le
\mathcal E_{\mathrm D}(h_{r,R})
+\lambda\sum_xh_{r,R}(x)^2\nu(x)
\le
\cC_{o,\rho}(r,R)+\lambda M_o(R).
\]
For $R\ge2r$, define the radial cutoff
\[
\eta_r(x)=
\begin{cases}
1,&d_\rho(o,x)\le r,\\
(2r-d_\rho(o,x))/r,&r<d_\rho(o,x)<2r,\\
0,&d_\rho(o,x)\ge2r.
\end{cases}
\]
It is admissible for $\cC_{o,\rho}(r,R)$.  Since the scalar cutoff is
$r^{-1}$-Lipschitz and
$|d_\rho(o,x)-d_\rho(o,y)|\le\rho(x,y)$ for $x\sim y$,
\eqref{eq:adapted-weight} gives
$\mathcal E_{\mathrm D}(\eta_r)\le r^{-2}M_o(2r)\le r^{-2}M_o(R)$.
This proves both cutoff estimates.
\end{proof}

\section{Proofs of the main results}

\subsection{Proof of the parabolic capacity estimate}

\begin{proof}[Proof of Theorem~\ref{thm:heat-capacity-lower}]
Fix the absolute constant $\eta:=1/2$ and apply
Proposition~\ref{prop:resolvent-capacity-lower}.  Corollary~\ref{cor:resolvent-reduction},
that proposition, and Tonelli's theorem give
\begin{align*}
\cH_{q,\sigma}(o)
&\ge
c_q\int_0^{\eta^2}\lambda^{q-2}L_\lambda^\nu(o)\,\dd\lambda\\
&\ge
c_q\int_0^{\eta^2}\lambda^{q-2}
\int_0^{\eta\lambda^{-1/2}}
 r\,\Cap_\lambda^\nu\bigl(B_\rho(o,r)\bigr)^{1-q}
 \,\dd r\,\dd\lambda.
\end{align*}
This proves \eqref{eq:heat-resolvent-capacity-lower}.

Interchanging the integrals and retaining only $r\ge1$ gives
\begin{equation*}
\cH_{q,\sigma}(o)
\ge
c_q\int_1^\infty r
\int_0^{\eta^2/r^2}
\lambda^{q-2}\Cap_\lambda^\nu\bigl(B_\rho(o,r)\bigr)^{1-q}
\,\dd\lambda\,\dd r.
\end{equation*}
Set
\[
        C:=\cC_{o,\rho}(r),
        \qquad
        M:=M_o(2r).
\]
For $s\ge0$, set
\[
        \Phi_q(s):=\int_0^s t^{q-2}(1+t)^{1-q}\,\dd t.
\]
Since $B_\rho(o,2r)$ is finite and the graph is connected, one has $C>0$.
Lemma~\ref{lem:massive-annular-capacity-comparison} and $1-q<0$ imply
\[
\Cap_\lambda^\nu\bigl(B_\rho(o,r)\bigr)^{1-q}
\ge
(C+\lambda M)^{1-q}.
\]
With $t=\lambda M/C$,
\begin{align*}
\int_0^{\eta^2/r^2}
\lambda^{q-2}(C+\lambda M)^{1-q}\,\dd\lambda
&=
M^{1-q}
\int_0^{\eta^2M/(r^2C)}
 t^{q-2}(1+t)^{1-q}\,\dd t\\
&=
M^{1-q}\Phi_q\!\left(\eta^2\frac{M}{r^2C}\right).
\end{align*}
Consequently,
\begin{equation}\label{eq:heat-capacity-volume-phi-lower}
\cH_{q,\sigma}(o)
\ge
c_q\int_1^\infty
rM_o(2r)^{1-q}
\Phi_q\!\left(
\eta^2\frac{M_o(2r)}{r^2\cC_{o,\rho}(r)}
\right)\,\dd r.
\end{equation}

By \eqref{eq:annular-capacity-cutoff},
\[
        s:=\frac{M_o(2r)}{r^2\cC_{o,\rho}(r)}\ge1.
\]
For $t\ge1$,
\[
2^{1-q}\frac1t
\le
t^{q-2}(1+t)^{1-q}
\le
\frac1t.
\]
Together with positivity on compact subintervals of $(0,\infty)$, this gives
\begin{equation*}
        \Phi_q(\eta^2s)
        \ge c_q\log(1+s),
        \qquad s\ge1.
\end{equation*}
Substitution proves \eqref{eq:heat-capacity-volume-lower}.
\end{proof}

\subsection{Proof of the mixed capacity--volume criterion}

\begin{proof}[Proof of Theorem~\ref{thm:capacity-volume-criterion}]
Condition \eqref{eq:capacity-volume-condition} and
Theorem~\ref{thm:heat-capacity-lower} imply
\[
        \int_0^\infty\sum_xp_\sigma(t,o,x)^q\nu(x)\,\dd t=\infty.
\]
The conclusion follows from Corollary~\ref{cor:heat-energy-liouville}.
\end{proof}

\subsection{Proof of the heat-energy volume lower bound}

\begin{proof}[Proof of Theorem~\ref{thm:heat-energy-lower}]
Theorem~\ref{thm:heat-capacity-lower} and
\eqref{eq:annular-capacity-cutoff} give
\begin{align*}
\cH_{q,\sigma}(o)
&\ge
c_q\int_1^\infty
\frac{r}{M_o(2r)^{q-1}}
\log\!\left(1+\frac{M_o(2r)}{r^2\cC_{o,\rho}(r)}\right)\,\dd r\\
&\ge
c_q\int_1^\infty
\frac{r}{M_o(2r)^{q-1}}\,\dd r.
\end{align*}
The change of variables $s=2r$ proves
\eqref{eq:main-heat-lower}.
\end{proof}

\subsection{Proof of the parabolic volume criterion}

\begin{proof}[Proof of Theorem~\ref{thm:volume-criterion}]
Under \eqref{eq:volume-condition}, the integral on the right-hand side of
\eqref{eq:main-heat-lower} diverges.  Therefore
\[
        \int_0^\infty\sum_xp_\sigma(t,o,x)^q\nu(x)\,\dd t=\infty,
\]
and Corollary~\ref{cor:heat-energy-liouville} gives the conclusion.
\end{proof}

\begin{proof}[Proof of Corollary~\ref{cor:unweighted-volume-criterion}]
Take $\sigma\equiv1$ and $\rho\equiv1$.  Then $\nu=\mu$,
$\Delta_\sigma=\Delta$, and $d_\rho$ is the graph distance.  For every
integer $n\ge1$ and every $r\in[n,n+1)$,
\[
        B(o,r)=B(o,n).
\]
Consequently,
\[
\int_1^\infty\frac{r}{\mu(B(o,r))^{q-1}}\,\dd r
=
\sum_{n=1}^\infty
\frac{n+\tfrac12}{\mu(B(o,n))^{q-1}},
\]
which shows that
\[
\int_1^\infty\frac{r}{\mu(B(o,r))^{q-1}}\,\dd r=\infty
\quad\Longleftrightarrow\quad
\sum_{n=1}^\infty\frac{n}{\mu(B(o,n))^{q-1}}=\infty.
\]
The assertion follows from Theorem~\ref{thm:volume-criterion}.
\end{proof}

\begin{proof}[Proof of Corollary~\ref{cor:unweighted-equation-volume-criterion}]
A positive global classical solution on $(0,\infty)\times V$ is a nonnegative
global classical supersolution of the corresponding inequality, so the first
assertion follows from Corollary~\ref{cor:unweighted-volume-criterion}.  If a
global Cauchy solution as in the second assertion existed, its restriction to
$(0,\infty)\times V$ would likewise be a nonnegative global classical
supersolution.  The same corollary would force this restriction to vanish
identically, and pointwise continuity at $t=0$ would then give $u_0\equiv0$,
contrary to the hypothesis.
\end{proof}

\section{Converse and sharpness examples}
\label{sec:sharpness}

\subsection{A heat-kernel mechanism for existence}

\begin{proof}[Proof of Proposition~\ref{prop:heat-kernel-existence}]
Set
\[
        b(t):=\sup_{x\in V}p_\sigma(t,o,x),
        \qquad t>0.
\]
The semigroup property, symmetry, and sub-Markov property give, for $s,t>0$,
\begin{align*}
p_\sigma(t+s,o,x)
&=\sum_{y\in V}
p_\sigma(t,o,y)p_\sigma(s,y,x)\nu(y)\\
&\le b(t)\sum_{y\in V}p_\sigma(s,x,y)\nu(y)
\le b(t).
\end{align*}
Moreover, symmetry and sub-Markovianity imply, for every $x\in V$,
\[
        p_\sigma(t,o,x)\nu(o)
        =p_\sigma(t,x,o)\nu(o)
        \le
        \sum_{y\in V}p_\sigma(t,x,y)\nu(y)
        \le1.
\]
Hence $b(t)\le\nu(o)^{-1}$.  Thus $b$ is finite and nonincreasing.  Put
\[
        h(t,x):=p_\sigma(t+1,o,x),
        \qquad
        b_0(t):=b(t+1),
        \qquad t\ge0.
\]
By \eqref{eq:heat-kernel-existence-condition},
\[
        \int_0^\infty b_0(t)^{q-1}\,\dd t<\infty.
\]
Extend $b_0$ to $(-\infty,0)$ by the constant value $b_0(0)$.  Let
$\varphi\in C_c^\infty((0,1))$ be nonnegative with
$\int_0^1\varphi(s)\,\dd s=1$, and define
\[
        B(t):=\int_0^1\varphi(s)b_0(t-s)\,\dd s,
        \qquad t\ge0.
\]
Then $B$ is smooth and $B(t)\ge b_0(t)$ because $b_0$ is nonincreasing.
Moreover, $B(t)\le b_0(t-1)$ for $t\ge1$, and hence
\[
        J:=\int_0^\infty B(t)^{q-1}\,\dd t<\infty.
\]

Choose $\varepsilon>0$ so small that
\[
        (q-1)\varepsilon^{q-1}J<1,
\]
and set
\[
A_\varepsilon(t)
:=
\left[
\varepsilon^{1-q}
-(q-1)\int_0^tB(s)^{q-1}\,\dd s
\right]^{-1/(q-1)}.
\]
This is a positive bounded $C^1$ function on $[0,\infty)$ and
\[
        A_\varepsilon'(t)
        =A_\varepsilon(t)^qB(t)^{q-1}.
\]
Since $h$ is positive and solves the heat equation, the function
\[
        U_\varepsilon(t,x):=A_\varepsilon(t)h(t,x)
\]
is a positive global classical supersolution.  Indeed,
\begin{align*}
(\partial_t-\Delta_\sigma)U_\varepsilon(t,x)
&=A_\varepsilon(t)^qB(t)^{q-1}h(t,x)\\
&\ge A_\varepsilon(t)^qh(t,x)^q
=U_\varepsilon(t,x)^q.
\end{align*}
Fix, for example, the shift $s=1$ and set
\[
        a_0:=\frac12U_\varepsilon(1,o)>0.
\]
The quantitative part of
Proposition~\ref{prop:supersolution-solution-equivalence} now gives, for every
$a\in(0,a_0]$, a positive global classical solution of the Cauchy problem
\eqref{eq:cauchy-problem} with $u_0=a\1_{\{o\}}$.
\end{proof}

\begin{proof}[Proof of Corollary~\ref{cor:heat-kernel-volume-existence}]
By \eqref{eq:anchored-heat-kernel-upper} and the change of variables
$t=r^2$,
\begin{align*}
\int_1^\infty
\left(\sup_{x\in V}p_\sigma(t,o,x)\right)^{q-1}\,\dd t
&\le
C^{q-1}\int_1^\infty\frac{\dd t}{M_o(\sqrt t)^{q-1}}\\
&=
2C^{q-1}\int_1^\infty
\frac{r}{M_o(r)^{q-1}}\,\dd r.
\end{align*}
The right-hand side is finite by
\eqref{eq:volume-convergence-condition}, so
Proposition~\ref{prop:heat-kernel-existence} applies.
\end{proof}

\subsection{The lattice example}

Consider $\mathbb Z^d$, $d\ge1$, with nearest-neighbor edges of unit
conductance, $\sigma\equiv1$, and $\rho\equiv1$.  Then
\[
        M_0(r)\asymp r^d,
        \qquad
        \sup_{x\in\mathbb Z^d}p(t,0,x)\lesssim t^{-d/2},
        \qquad t\ge1.
\]
The heat-kernel estimate is standard; see, for example,
\cite{Delmotte,GT01}.  The corresponding semilinear
existence--nonexistence dichotomy on $\mathbb Z^d$ was also studied in
\cite{LinWu17}.  Consequently,
\[
\int_1^\infty\frac{r}{M_0(r)^{q-1}}\,\dd r
\begin{cases}
=\infty,&q\le1+\dfrac2d,\\[4pt]
<\infty,&q>1+\dfrac2d.
\end{cases}
\]
Theorem~\ref{thm:volume-criterion} excludes every nontrivial nonnegative
Cauchy datum when $q\le1+2/d$---indeed, it excludes even trace-free
supersolutions---whereas Corollary~\ref{cor:heat-kernel-volume-existence}
produces positive global Cauchy solutions for all sufficiently small
point-source data when $q>1+2/d$.  Thus the two results recover both sides of
the classical Fujita threshold on $\mathbb Z^d$.

\subsection{Logarithmic and iterated-logarithmic thresholds}

We next construct a family for which the polynomial growth exponent is
critical and lower-order volume factors determine the conclusion.  Fix
$D>1$ and $\beta,\gamma\in\mathbb R$.  Choose $n_0>e^e$ sufficiently large
that the sequence
\begin{equation}\label{eq:half-line-conductances}
        a_n
        :=
        (n+n_0)^{D-1}
        \bigl(\log(n+n_0)\bigr)^\beta
        \bigl(\log\log(n+n_0)\bigr)^\gamma,
        \qquad n\in\mathbb N_0,
\end{equation}
is increasing.  On $V=\mathbb N_0$, join $n$ to $n+1$ with conductance
$a_n$.  Thus
\[
        \mu(0)=a_0,
        \qquad
        \mu(n)=a_{n-1}+a_n,\quad n\ge1.
\]
Take $\sigma\equiv1$ and $\rho(n,n+1)=1$.

\begin{proposition}\label{prop:half-line-heat-kernel}
For the weighted half-line defined above, the unit edge length
$\rho\equiv1$ is $\mu$-adapted, its associated graph distance is proper,
\begin{equation}\label{eq:half-line-volume}
        M_0(r)
        \asymp
        r^D(\log r)^\beta(\log\log r)^\gamma,
        \qquad r\to\infty,
\end{equation}
and
\begin{equation}\label{eq:half-line-heat-kernel-upper}
        \sup_{n\in\mathbb N_0}p(t,0,n)
        \le\frac{C}{M_0(\sqrt t)},
        \qquad t\ge1.
\end{equation}
\end{proposition}

\begin{proof}
Adaptedness holds with equality:
\[
        \sum_{m\sim n}\mu_{nm}\rho(n,m)^2=\mu(n).
\]
The metric is proper because its balls are finite intervals.  Since
$a_{n+1}/a_n\to1$, adjacent conductances are uniformly comparable.
In particular,
\[
        \frac{\mu_{n,n+1}}{\mu(n)}
        \quad\text{and}\quad
        \frac{\mu_{n,n+1}}{\mu(n+1)}
\]
are bounded below by a positive constant.  Summation of
\eqref{eq:half-line-conductances} gives
\[
        \sum_{n=0}^{\lfloor r\rfloor}a_n
        \asymp
        r^D(\log r)^\beta(\log\log r)^\gamma,
\]
which proves \eqref{eq:half-line-volume}.

We record the two geometric estimates needed for the heat-kernel bound.  The
measure $\mu$ is uniformly doubling.  If $n\ge4r$, all vertex weights on
$B(n,2r)$ are comparable, and doubling follows as on an ordinary path.  If
$n<4r$, then $B(n,2r)\subset[0,6r]$.  When $n\le r$, the ball $B(n,r)$
contains $[0,r]$; when $r<n<4r$, it contains $[n,n+r]$.  In both cases,
\eqref{eq:half-line-volume} gives
\[
        \mu(B(n,2r))\le C\mu(B(n,r)).
\]
The remaining bounded radii are controlled by the uniform comparability of
adjacent vertex weights.

The scale-invariant Poincar\'e inequality follows from the
one-dimensional weighted Hardy criterion; see \cite{Miclo99}.  For
completeness, let $I=[A,B]\cap\mathbb N_0$ and choose
\[
        m:=\left\lfloor\frac{A+B}{2}\right\rfloor.
\]
Since the weighted average minimizes
$c\mapsto\sum_{n\in I}\mu(n)|f(n)-c|^2$, it suffices to estimate
$\sum_{n\in I}\mu(n)|f(n)-f(m)|^2$.  The two one-sided Hardy
inequalities reduce this estimate, up to a universal factor, to the maximum
of
\begin{equation}\label{eq:half-line-hardy-quantities}
\sup_{A\le j<m}
\mu([A,j])\sum_{k=j}^{m-1}\frac1{a_k},
\qquad
\sup_{m<j\le B}
\mu([j,B])\sum_{k=m}^{j-1}\frac1{a_k}.
\end{equation}
Since $a_k$ is increasing, the first quantity is bounded by
\[
        C\sup_{A\le j<m}(j-A+1)(m-j)
        \le C(B-A+1)^2.
\]
The vertex weights are also increasing.  Since
$m\ge(A+B-1)/2$, regular variation gives $a_B/a_m\le C$, and the second
quantity in \eqref{eq:half-line-hardy-quantities} is bounded by
\[
        C\sup_{m<j\le B}
        (B-j+1)(j-m)\frac{a_B}{a_m}
        \le C(B-A+1)^2.
\]
Taking $I=B(n,r)$, we have $B-A+1\le2r+1$.  Hence, for $r\ge1$,
\[
\sum_{x\in B(n,r)}
|f(x)-f_{B(n,r)}|^2\mu(x)
\le
Cr^2
\sum_{\substack{k,k+1\in B(n,r)}}
a_k|f(k+1)-f(k)|^2.
\]
This is stronger than the Poincar\'e inequality required in
\cite{Delmotte}; the case $r<1$ is trivial.  Thus the volume-doubling
property, the scale-invariant Poincar\'e inequality, and the uniform lower
bound on transition probabilities all hold.

Let
\[
        V_\mu(n,r):=\mu(B(n,r)).
\]
The preceding estimates verify volume doubling, the scale-invariant
Poincar\'e inequality, and Delmotte's edge-ellipticity condition: there exists
$\alpha>0$ such that
\[
        \mu_{nm}\ge\alpha\mu(n)
        \qquad\text{whenever }m\sim n.
\]
Since we work in continuous time, no loop or holding-probability condition is
required; see the discussion following Definition~1.3 in \cite{Delmotte}.
Delmotte's continuous-time upper estimate therefore gives
\[
        P_t(0,n)
        \le
        \frac{C\mu(n)}
        {\sqrt{V_\mu(0,\sqrt t)V_\mu(n,\sqrt t)}},
        \qquad t\ge1,
\]
after discarding the off-diagonal decay factor.

Since the vertex weights $\mu(n)$ are increasing,
\[
        V_\mu(n,r)\ge V_\mu(0,r)=M_0(r)
        \qquad(n\in\mathbb N_0,\ r\ge0).
\]
Indeed, writing $N=\lfloor r\rfloor$, if $n\le N$, then $B(n,r)$
contains $[0,N]$; if $n>N$, then it contains $[n,n+N]$, whose
$\mu$-measure is at least that of $[0,N]$.  Consequently,
\[
        P_t(0,n)
        \le
        \frac{C\mu(n)}{M_0(\sqrt t)}.
\]
Since $P_t(0,n)=p(t,0,n)\mu(n)$ in the present heat-kernel convention,
division by $\mu(n)$ proves \eqref{eq:half-line-heat-kernel-upper}.
\end{proof}

Set
\[
        q_D:=1+\frac2D.
\]
At this polynomially critical exponent, \eqref{eq:half-line-volume} gives
\[
\int_{e^e}^\infty\frac{r}{M_0(r)^{q_D-1}}\,\dd r
\asymp
\int_{e^e}^\infty
\frac{\dd r}{
r(\log r)^{2\beta/D}(\log\log r)^{2\gamma/D}}.
\]
The last integral diverges precisely when
\begin{equation}\label{eq:half-line-divergence-range}
        \beta<\frac D2,
        \qquad\text{or}\qquad
        \beta=\frac D2\ \text{and}\ \gamma\le\frac D2,
\end{equation}
and converges precisely when
\begin{equation}\label{eq:half-line-convergence-range}
        \beta>\frac D2,
        \qquad\text{or}\qquad
        \beta=\frac D2\ \text{and}\ \gamma>\frac D2.
\end{equation}
Accordingly, Theorem~\ref{thm:volume-criterion} gives trace-free
nonexistence, and hence Cauchy nonexistence for every nontrivial nonnegative
initial datum, in \eqref{eq:half-line-divergence-range}.  In
\eqref{eq:half-line-convergence-range},
Proposition~\ref{prop:half-line-heat-kernel} and
Corollary~\ref{cor:heat-kernel-volume-existence} give positive global Cauchy
solutions for all sufficiently small point-source data.

\begin{remark}
Take $D=4$ and $q=q_D=3/2$.  For $\beta=2$,
\[
        M_0(r)\asymp
        r^4(\log r)^2(\log\log r)^\gamma,
\]
and the critical integral is comparable to
\[
        \int_{e^e}^\infty
        \frac{\dd r}{r\log r(\log\log r)^{\gamma/2}}.
\]
Thus $\gamma=2$ gives trace-free nonexistence, while $\gamma=3$ gives
positive global Cauchy solutions for sufficiently small point-source data.
The two examples differ only in the iterated-logarithmic exponent.  Likewise,
with $\gamma=0$, the logarithmic threshold is $\beta=2$: $\beta=2$ gives
nonexistence, whereas, for example, $\beta=3$ gives small-data Cauchy
existence.  Finally,
\[
        \sum_{n=0}^\infty\frac1{a_n}<\infty
\]
for every $\beta,\gamma$ in this $D=4$ family.  Hence all these networks are
transient; the logarithmic transition is not a disguised
recurrence--transience transition.
\end{remark}

\end{document}